\documentclass[11pt,a4paper,reqno]{amsart}
\usepackage[english]{babel}
\usepackage[T1]{fontenc}
\usepackage{verbatim}
\usepackage{palatino}
\usepackage{amsmath}
\usepackage{mathabx}
\usepackage{amssymb}
\usepackage{amsthm}
\usepackage{amsfonts}
\usepackage{graphicx}
\usepackage{esint}
\usepackage{color}
\usepackage{mathtools}
\usepackage{overpic}
\usepackage{soul}

\usepackage[colorlinks = true, citecolor = black]{hyperref}
\author{Borys Kuca, Tuomas Orponen, and Tuomas Sahlsten}
\title{On a continuous S\'ark{\"o}zy type problem}
\address{Department of Mathematics and Statistics\\ University of Jyv\"askyl\"a,
P.O. Box 35 (MaD)\\
FI-40014 University of Jyv\"askyl\"a\\
Finland}
\email{tuomas.t.orponen@jyu.fi}
\email{bjkuca@jyu.fi}

\address{Department of Mathematics and Systems Analysis \\ Aalto University \\ P.O. Box 11100 \\ FI-00076 Aalto \\ Finland.}
\email{tuomas.sahlsten@aalto.fi}

\thanks{B.K. and T.O. are supported by the Academy of Finland via the projects \emph{Quantitative rectifiability in Euclidean and non-Euclidean spaces} and \emph{Incidences on Fractals}, grant Nos. 309365, 314172, 321896. T.S. was supported by a start-up grant from the University of Manchester.}

\newcommand{\R}{\mathbb{R}}

\newcommand{\N}{\mathbb{N}}

\newcommand{\Z}{\mathbb{Z}}
\newcommand{\NN}{\mathbb{N}}
\newcommand{\ZZ}{\mathbb{Z}}
\newcommand{\FF}{\mathbb{F}}
\newcommand{\CC}{\mathbb{C}}

\newcommand{\spt}{\operatorname{spt}}
\newcommand{\Hd}{\dim_{\mathrm{H}}}

\newcommand{\diam}{\operatorname{diam}}

\newcommand{\dist}{\operatorname{dist}}

\newcommand{\eps}{\varepsilon}

\def\Barint_#1{\mathchoice
          {\mathop{\vrule width 6pt height 3 pt depth -2.5pt
                  \kern -8pt \intop}\nolimits_{#1}}%
          {\mathop{\vrule width 5pt height 3 pt depth -2.6pt
                  \kern -6pt \intop}\nolimits_{#1}}%
          {\mathop{\vrule width 5pt height 3 pt depth -2.6pt
                  \kern -6pt \intop}\nolimits_{#1}}%
          {\mathop{\vrule width 5pt height 3 pt depth -2.6pt
                  \kern -6pt \intop}\nolimits_{#1}}}

\numberwithin{equation}{section}

\theoremstyle{plain}
\newtheorem{thm}{Theorem}
\newtheorem*{"thm"}{"Theorem"}

\newtheorem{lemma}{Lemma}

\newtheorem{proposition}{Proposition}
\newtheorem*{question}{Question}

\theoremstyle{definition}

\theoremstyle{remark}

\newcommand{\nref}[1]{(\hyperref[#1]{#1})}

\DeclareMathSymbol{\intop}  {\mathop}{mathx}{"B3}

\begin{document} 

\begin{abstract} We prove that there exists a constant $\epsilon > 0$ with the following property: if $K \subset \R^2$ is a compact set which contains no pair of the form $\{x, x + (z, z^{2})\}$ for $z \neq 0$, then $\Hd K \leq 2 - \epsilon$.  \end{abstract}

\maketitle


\section{Introduction} 

A classical problem originating in the works of Furstenberg \cite{MR498471} and S\'ark{\"o}zy \cite{Sarkozy, sarkozy_1978b} asks to study the size of sets $A \subset \{1,\dots,N\}$ avoiding pairs of the form $\{x,  x+z^2\}$. Furstenberg and S\'ark{\"o}zy independently proved that such sets cannot have positive density, as $N \to \infty$, and S\'ark{\"o}zy gave quantitative estimates on the size of $|A|$ in terms of $N$. The purpose of this article is to consider a continuous analogue of this problem for sets of fractional dimension. S\'ark{\"o}zy's original formulation does not make sense on $\R$, because every pair of points $x < y$ can be expressed as a progression $\{x,  x + z^2\}$. On $\R^{2}$, however, a natural variant of the problem asks for the existence of configurations of the form $\{x,  x + (z,z^{2})\}$. We are able to provide the following answer:

\begin{thm}\label{mainThm} There exists an absolute constant $\epsilon > 0$ such that the following holds. Let $K \subset \R^{2}$ be a compact set with Hausdorff dimension $\Hd K \geq 2 - \epsilon$. Then, there exist $x \in K$ and $z \neq 0$ such that $x + (z,z^{2}) \in K$.  \end{thm}

The value of $\epsilon > 0$ could, in principle, be deduced from the proof, but it would be rather small. We suspect that Theorem \ref{mainThm} is true for all $\epsilon < \tfrac{1}{2}$. For $\epsilon \geq \tfrac{1}{2}$ there exist counterexamples: a $\tfrac{1}{2}$-H\"older graph with H\"older constant $<1$ over the $y$-axis whose graph has Hausdorff dimension $\tfrac{3}{2}$ cannot contain configurations $\{x,  x + (z,z^{2})\}$. The main point of Theorem \ref{mainThm} is the presence of "non-linearity" in the term $(z,  z^{2})$. A similar conclusion fails if $(z,z^{2})$ is replaced by a linear term, say $(z,  z)$ or $(z,  0)$: this is simply because there exist graphs of Hausdorff dimension $2$.

\subsection{Related work} 
The Furstenberg-S\'ark\"ozy theorem on sets lacking the configuration $\{x,  x+z^2\}$ has sparked a tremendous amount of research. One widely investigated question has been to prove bounds for subsets lacking the progression $\{x,  x+z^2\}$ or more generally $\{x,  x+P(z)\}$ for a fixed choice of polynomial $P\in\ZZ[z]$ with zero constant term\footnote{We note that if $P(z) = z^2 + 1$, then the set $3\NN$ of upper density $1/3$ contains no progressions $\{x,  x + P(z)\}$ as $P(z)$ is never divisible by 3 for $z\in\ZZ$. The condition that $P$ has no constant term ensures that this obstruction does not happen. It can be replaced by a more general condition of $P$ being \emph{intersective}, see e.g. \cite{rice_2019} for a discussion.} \cite{balog_pelikan_pintz_szemeredi_1994, lucier_2006, rice_2019, slijepcevic_2003}. The best bound of the form $|A|\leq C N/(\log N)^{\eps \log\log\log N}$ for sets $A\subset \{1, ..., N\}$ lacking $\{x,  x+z^2\}$ comes from Bloom and Maynard \cite{bloom_maynard_2020}. In another direction, the Furstenberg-S\'ark\"ozy theorem has been substantially extended by Bergelson and Leibman \cite{MR1325795}, who showed that all subsets of integers with positive upper Banach density contain configurations of the form
\begin{align}\label{polynomial progressions}
    \{x,  x+P_1(z),  ...,  x+P_t(z)\}
\end{align}
with $z\neq 0$ for (fixed) distinct polynomials $P_1, ..., P_t\in\ZZ[z]$ with zero constant terms. Gowers then posed the question of finding explicit bounds in the Bergelson-Leibman theorem \cite[Problem 11.4]{Gowers}. Studying sizes of sets avoiding \textit{non-linear} patterns in integers and finite fields has become a widely active topic ever since, with a number of new results being proved recently. Early works of Bourgain and Chang \cite{bourgain_chang_2017}, Dong, Li and Sawin \cite{dong_li_sawin_2017} as well as Peluse \cite{peluse_2018} provided upper bounds of the form $|A| \leq C p^{1 - \eps}$ for subsets of finite fields $\FF_q$ with $p$ prime lacking $$\{x,  x+P_1(z),  x+P_2(z)\}$$ for fixed linearly independent polynomials $P_1, P_2\in\ZZ[z]$, the simplest example of which is the progression $\{x,  x+z,  x+z^2\}$. Around the same time, Prendiville proved a bound $|A|\leq C/(\log\log N)^{\eps}$ for subsets of $\{1, ..., N\}$ lacking (\ref{polynomial progressions}) with $P_i(z) = a_i z^k$ for fixed $k\geq 2$ and distinct $a_i\in\ZZ$. A recent breakthrough is the work of Peluse \cite{peluse_2019}, in which bounds of the form $|A|\leq C p^{1-\eps}$ have been obtained for subsets of $\FF_q$ lacking (\ref{polynomial progressions}) with any fixed distinct linearly independent polynomials $P_1, ..., P_t$. The generality of Peluse's result comes from the replacement of deep tools from number theory or algebraic geometry used in earlier works \cite{bourgain_chang_2017, dong_li_sawin_2017} by a more robust Fourier analytic argument that has allowed her to tackle longer progressions than before. The degree-lowering argument invented in \cite{peluse_2019} has been successfully adapted to a number of different settings. In the integer setting, it has yielded bounds $|A| \leq C N/(\log N)^{\eps}$ for subsets of $\{1, ..., N\}$ lacking $\{x,  x + z,  x + z^{2}\}$ with $z > 0$ \cite{peluse_prendiville_2019, peluse_prendiville_2020} as well as estimates $|A| \leq C N/(\log \log N)^{\eps}$ for subsets lacking (\ref{polynomial progressions}) with distinct-degree polynomials \cite{peluse_2020}. In the finite-field setting, it has led to bounds for subsets of $\FF_q$ lacking the progression
\begin{align*}
    \{x,  x+z,  ...,  x+(t-1)z,  x + z^t,  ...,  x + z^{t+k-1}\}
\end{align*}
for fixed $t, k\in\N_+$ of essentially the same shape as the bounds for subsets lacking just $t$-term arithmetic progressions \cite{kuca_2020}. 

An analogous question has also been examined for multidimensional configurations. Han, Lacey, and Yang have found non-trivial patterns of the form $$\{x,  x + (P_{1}(z),  0),  x + (0,  P_{2}(z))\}$$ inside sets $A \subset \FF_{p}^2$ with $|A| \geq p^{2 - 1/16}$, and where $P_{1}, P_{2}$ are two fixed polynomials of distinct degrees \cite{MR4304416}. Their result has been generalised by the first author, who has proved that all subsets $A\subset\FF_q^d$ lacking the progression
\begin{align*}
    \{x,  x + v_1 P_1(z),  ...,  x + v_t P_t(z)\}
\end{align*}
for fixed direction vectors $v_1, ..., v_t \in \ZZ^d$ and distinct-degree polynomials $P_1, ..., P_t\in\ZZ[z]$ have size at most $|A|\leq C p^{d-\eps}$ \cite{kuca_2021}. The methods used in finite fields are completely different from the methods of Theorem \ref{mainThm}. To highlight this difference, in Section \ref{section on finite fields} we also give the proof of the following analogue of Theorem \ref{mainThm} in finite fields:
\begin{thm}\label{finite field bound}
Let $q = p^n$ be a prime power for a prime $p>2$, and let $\FF_q$ be the finite field with $q$ elements. Suppose that $A\subset\FF_q^2$ has cardinality $|A|\geq 2 q^\frac{3}{2}$. Then, there exists $x \in A$ and $z \neq 0$ such that $x + (z,  z^2) \in A$.
\end{thm}
We remark that a weaker version of Theorem \ref{finite field bound} with bound $q^{2-\epsilon}$ for some $\eps > 0$ can be deduced from \cite{MR4304416,kuca_2020}. Theorem \ref{finite field bound} allows for a much more succinct proof than Theorem \ref{mainThm} thanks to the availability of standard estimates of quadratic sums over finite fields. Moreover, the exponent $\frac{3}{2} = 2 -\frac{1}{2}$ that we get in Theorem \ref{finite field bound} corresponds to $\eps = \frac{1}{2}$, which - as indicated before - we expect to be the supremum of the values of $\eps$ for which Theorem \ref{mainThm} works. See Section \ref{s:acknowledgements} for further remarks.

The question of finding the optimal bound in Theorem \ref{finite field bound} is wide open, but there are reasons to suspect that the bound in Theorem \ref{finite field bound} is not optimal, at least when $q$ is prime. In the single dimensional version of the problem, a Fourier analytic argument similar to one used in the proof of Theorem \ref{finite field bound} shows that for every prime power $q$, all subsets of $\FF_q$ of cardinality at least $q^{1/2}$ contain points $x, x+z^2$ for some $x,z\in\FF_q$ with $z\neq 0$. At the same time, it is known that for infinitely many primes $p$, there exists a subset $A\subseteq\FF_p$ of cardinality $|A|\gtrsim \log p \log\log\log p$ lacking such a configuration \cite{graham_ringrose_1990}. Taking the preimage $B=\FF_p \times A$ of this set $A$ inside $\FF_p^2$, it follows that for infinitely many primes $p$, there exists a set $|B|\gtrsim p \log p \log\log\log p$ lacking $x, x+(z, z^2)$ for any $z\neq 0$. This shows a big disparity between the upper and lower bounds, and it is quite likely that the lower bound, coming from an explicit construction, is closer to the truth.

In parallel with the study of Szemer\'edi and Roth type theorems for subsets of integers and finite fields, in articles \cite{MR853455,MR942826} Bourgain initiated the study of these questions for subsets of the real line. Theorem \ref{mainThm} can be viewed as a relative of the "non-linear Roth theorem" of Bourgain \cite{MR942826} from 1988: a Lebesgue measurable set $K \subset [0,1]$ with measure $\mathcal{L}^{1}(K) = \epsilon > 0$ contains a triple $\{x,  x + z,  x + z^{2}\}$ for some $|z| \geq \delta(\epsilon) > 0$. In Bourgain's proof, however, $\delta(\epsilon)$ decays super-exponentially fast as $\epsilon \to 0$. This means that the assumption $\mathcal{L}^{1}(K) = \epsilon$ cannot be easily relaxed to $\Hd K \geq 1 - \epsilon$. The same is true in a recent work of Durcik, Guo, and Roos \cite{MR3939567}, where the authors generalise Bourgain's theorem to progressions of the form $\{x,  x + z,  x + P(z)\}$, for any polynomial $P$ of degree $\geq 2$. Following Bourgain's approach, we could obtain a version of Theorem \ref{mainThm}, where we assume $\mathcal{L}^{2}(K) \geq \epsilon > 0$, and conclude the existence of $\{x,  x + (z,z^{2})\} \subset K$ with $|z| \geq \delta(\epsilon)$. Indeed, the key novelty of Theorem \ref{mainThm} is to relax the assumption $\mathcal{L}^{2}(K) \geq \epsilon$ to $\Hd K \geq 2 - \epsilon$. 

There are also several existing works in the fractal regime, that is, for sets $K \subset \R^{d}$ with $\Hd K < d$, see for example the papers \cite{MR2545245} of {\L}aba and Pramanik, Chan and {\L}aba \cite{MR3481177}, Henriot, {\L}aba, and Pramanik \cite{MR3531369}, and Fraser, Guo, and Pramanik \cite{2019arXiv190411123F}. In each of these papers, the authors prove the existence of progressions (of various kinds) in subsets $K \subset \R^{d}$ of dimension $\Hd K \geq d - \epsilon$, which additionally support a measure with sufficiently rapid Fourier decay (therefore the \emph{Fourier dimension} of $K$ needs to be positive). The paper \cite{2019arXiv190411123F}, in particular, proves the existence of triplets $\{x,  x + z,  x + P(z)\}$, where $P$ is a polynomial without constant term, and degree at least $2$ under positive Fourier dimension assumption. We emphasise that in Theorem \ref{mainThm} we do not impose any Fourier dimension assumption on the set $K$. In a different direction, Yavicoli \cite{2019arXiv191010057Y} proves the existence of arithmetic progressions (and more general configurations) in subsets of $\R$ with sufficiently high \emph{thickness}; notably, even sets with dimension less than one can have sufficiently high thickness for her results to apply. Falconer and Yavicoli \cite{2021arXiv210201186F} have thereafter obtained analogous results in $\R^{d}$.

We also mention a recent paper \cite{MR4295087} of Christ, Durcik, and Roos, where the authors find patterns of the form $\{(x,  y),  (x + z,  y),  (x,  y + z^{2})\} \subset K$ inside Lebesgue measurable sets $K \subset \R^{2}$ with $\mathcal{L}^{2}(K) \geq \epsilon > 0$. The authors also prove the gap estimate $|z| > \exp(-\exp(\epsilon^{-C}))$. This result is relevant for Theorem \ref{mainThm}, because $(x,  y + z^{2}) = (x + z,  y) + (-z,  (-z)^{2})$. In other words, the case $\mathcal{L}^{2}(K) \geq \epsilon > 0$ of Theorem \ref{mainThm} would follow immediately from \cite{MR4295087} (it should be emphasised, however, that the case $\mathcal{L}^{2}(K) \geq \epsilon$ of Theorem \ref{mainThm} is easy, whereas finding the "corners" $\{(x,  y),  (x + z,  y),  (x,  y + z^{2})\}$ requires rather sophisticated harmonic analysis technology).

Finally, there is also a large body of literature of constructing surprisingly large subsets of $\R^{d}$ which avoid various patterns and progressions. Keleti \cite{MR2431353} has constructed subsets of full dimension on $\R$ which contain no $3$-term arithmetic progressions (or more generally similar copies of an \emph{a priori} fixed set with $3$ elements). For extensions and generalisations, and more results in this direction, see the papers \cite{MR3016405} of Maga, \cite{MR3672917} of M\'ath\'e, \cite{MR4238597} of Denson, {\L}aba, and Zahl, and \cite{MR4305962} of Yavicoli. Quite surprisingly, one may even construct subsets of $\R$ with full Hausdorff dimension \textit{and} Fourier dimension which avoid $3$-term arithmetic progressions: this counterintuitive construction is due to Shmerkin \cite{MR3658188}. All the constructions cited here are quite delicate, and "random" fractals would not work: for finding patterns in random fractals, see the paper \cite{MR4101330} by Shmerkin and Suomala. 


\subsection{Proof outline} The following proposition gives a sufficient condition for checking the validity of Theorem \ref{mainThm}. The proof is short and standard, so we give it right away.

\begin{proposition}\label{auxProp} Let $K \subset \R^{2}$ be compact, and let $\mu$ be a Radon probability measure with $\spt(\mu) \subset K$. Let $\psi \in \mathcal{S}(\R^{2})$ be a non-negative Schwartz function, and let $\pi$ be a Radon measure supported on $\{(z,z^{2}) : |z| > 0\}$. Assume that
\begin{displaymath} \liminf_{\delta \to 0} \int (\mu \ast \psi_{\delta}) \ast \pi \, d\mu > 0, \end{displaymath}
where $\psi_{\delta}(x) := \delta^{-2}\psi(x/\delta)$. Then, there exists $x \in K$ and $z \neq 0$ such that $x + (z,z^{2}) \in K$.
\end{proposition}

\begin{proof} Let $K,\mu,\psi,\pi$ be the objects in Proposition \ref{auxProp}. Write 
\begin{displaymath} L := \liminf_{\delta \to 0} \int (\mu \ast \psi_{\delta}) \ast \pi \, d\mu > 0, \end{displaymath}
and let $\delta > 0$ be so small that $\int (\mu \ast \psi_{\delta}) \ast \pi \, d\mu > L/2$. Then, since $\mu$ is a probability measure, there exists a point $x_{\delta} \in \spt (\mu) \subset K$ such that $(\mu_{\delta} \ast \pi)(x_{\delta}) > L/2$. Since convolution is associative and commutative, we have 
\begin{displaymath} (\mu \ast \psi_{\delta}) \ast \pi = \mu \ast \pi_{\delta} \quad \text{with} \quad \pi_{\delta} = \pi \ast \psi_{\delta}. \end{displaymath}
Hence
\begin{align} L/2 < (\mu \ast \pi_{\delta})(x_{\delta}) = \int \pi_{\delta}(x_{\delta} - y) \, d\mu(y) & = \int_{\{\dist(x_{\delta} - y,\spt (\pi)) > \sqrt{\delta}\}} \pi_{\delta}(x_{\delta} - y) \, d\mu(y) \notag\\
&\label{form38} + \int_{\{\dist(x_{\delta} - y,\spt (\pi)) \leq \sqrt{\delta}\}} \pi_{\delta}(x_{\delta} - y) \, d\mu(y). \end{align}
We claim that the first integral tends to zero as $\delta \to 0$. Indeed, for every $z := x_{\delta} - y$ with $\dist(z,\spt (\pi)) > \sqrt{\delta}$, we have $|z - w| \geq \sqrt{\delta}$ for all $w \in \spt (\pi)$, and hence
\begin{displaymath} \pi_{\delta}(z) \leq \delta^{-2} \int \psi((z - w)/\delta) \, d\pi(w) \lesssim \delta^{-2} \int \left(\frac{\delta}{|z - w|} \right)^{6} \, d\pi(w) \leq \delta, \end{displaymath}
using the rapid decay of $\psi \in \mathcal{S}(\R^{2})$. Consequently also
\begin{displaymath} \int_{\{\dist(x_{\delta} - y,\spt (\pi)) > \sqrt{\delta}\}} \pi_{\delta}(x_{\delta} - y) \, d\mu(y) \lesssim \delta. \end{displaymath}
For $\delta$ much smaller than $L/4$, it now follows that $\eqref{form38} \geq L/4$, and consequently we may find a point $y_{\delta} \in \spt (\mu) \subset K$ such that
\begin{displaymath} \dist(x_{\delta} - y_{\delta},\spt (\pi)) \leq \sqrt{\delta}. \end{displaymath}
For $\delta > 0$ sufficiently small, we have 
\begin{equation}\label{form39} |x_{\delta} - y_{\delta}| \geq \dist(0,\spt (\pi)) - \dist(x_{\delta} - y_{\delta},\spt (\pi)) \geq \tfrac{1}{2}\dist(0,\spt (\pi)) > 0. \end{equation}
Now $K \times K \subset \R^{4}$ is compact, so we may choose a subsequence of $(x_{\delta},y_{\delta})_{\delta > 0}$ which converges to a point $(x,y) \in K \times K$. We have $x \neq y$ by \eqref{form39}. Moreover, 
\begin{displaymath} \dist(x - y,\spt (\pi)) = \lim_{\delta \to 0} \dist(x_{\delta} - y_{\delta},\spt (\pi)) = 0. \end{displaymath}
Hence $x = y + \gamma$ for some $\gamma \in \spt (\pi) \subset \{(z,z^{2}) : z \in \R\}$. This is what we claimed. \end{proof}

The proof of Theorem \ref{mainThm} will proceed as follows. In Section \ref{sec:gap}, we give a sufficient condition for the integral in Proposition \ref{auxProp} to be positive, see Lemma \ref{lemma1}. The key condition is that $\mu$ has a sufficiently large \textit{"spectral gap"}: this roughly means that $|\widehat{\mu}(\xi)| \leq \epsilon$ for all $\xi \in \R^{2}$ with $A \leq |\xi| \leq B$, but one has to be careful in choosing the constants $A,B,\epsilon$ in the right order. It is certainly well-known that such a hypothesis on $\mu$ would be useful in proving Theorem \ref{mainThm}. The main novelty of the paper lies in Section \ref{s:construction}, where a suitable measure $\mu$ is eventually constructed, see Proposition \ref{mainProp}. The idea is not to construct $\mu$ directly on $K$, but rather find a suitable dyadic box $Q$ where the density of $K$ is sufficiently high, construct $\mu$ inside $K \cap Q$, and renormalise.

This approach has one problem: the parabola $\{(z,z^{2}) : z \in \R\}$ looks increasingly flat inside small dyadic boxes, so our problem is not "invariant under renormalisation". To solve this, we consider parabolic dyadic boxes $Q$ instead of Euclidean ones. Then, the problem of finding pairs $x,x + (z,z^{2})$ inside $\R^{2}$ is equivalent to finding such pairs inside $K \cap Q$, and now the increased density of $K$ inside $Q$ becomes truly helpful. 

A small price to pay for using parabolic boxes is that the natural assumption of $K$ is no longer $\Hd K \geq 2 - \epsilon$, but rather $\Hd^{\Pi} K \geq 3 - \epsilon$, where $\Hd^{\Pi} K$ stands for the Hausdorff dimension of $K$ in the parabolic metric. Fortunately, $\Hd^{\Pi} K \geq 2\Hd K - 1$, so our initial assumption $\Hd K \geq 2 - \epsilon$ implies $\Hd^{\Pi} K \geq 3 - 2\epsilon$. The proof of Theorem \ref{mainThm} is finally concluded in Section \ref{sec:conclusion}. 

Given the connection to the parabolic metric, the following question seems reasonable:
\begin{question} Let $K \subset \R^{2}$ be a compact set with $\Hd^{\Pi} K > 2$. Does there exists a point $x \in K$ such that $x + (z,z^{2}) \in K$ for some $z \neq 0$? \end{question}

\subsection{Notations}


Basic notation we use throughout the article include:
\begin{itemize}
\item For $A,B \geq 0$, the notation $A \lesssim B$ means there exists an absolute constant $C > 0$ such that $A \leq C B$. If we allow the constant $C$ to depend on a parameter "$p$", we indicate this in the notation by writing $A \lesssim_{p} B$. We also let $A \sim B$ if $B\lesssim A \lesssim A$ and $A \sim_p B$ if $B\lesssim_p A \lesssim_p B$. 
\item $\Hd K$ is the Euclidean Hausdorff dimension of a set $K \subset \R^{d}$ (see \cite{zbMATH01249699} for a definition), and $\mathcal{L}^{d}$ is the $d$-dimensional Lebesgue measure on $\R^{d}$. We will also encounter the parabolic Hausdorff dimension $\Hd^{\Pi} K$ for $K \subset \R^{2}$, see Section \ref{s:construction}.
\item We write $\widehat{\mu}(\xi) = \int e^{-2\pi i x \cdot \xi} d\mu(x)$, $\xi \in \R^{d}$, for the Fourier transform of a finite Borel measure $\mu$ on $\R^{d}$. We also denote $\|\mu\|$ the total variation norm of a Borel measure $\mu$ on $\R^d$, which for non-negative measures equals to $\mu(\R^d)$
\item $\mathcal{S}(\R^{d})$ is the family of all Schwartz functions on $\R^{d}$.
\end{itemize}

\subsection{Acknowledgements}\label{s:acknowledgements} We are grateful to Sebastiano Nicolussi Golo for pointing out a (fortunately harmless) gap in our initial proof of Lemma \ref{lemma1}. We also thank Pham Van Thang for suggesting a shorter proof of Theorem \ref{finite field bound} with better bounds. Finally, we thank Alex Iosevich for pointing out Theorem \ref{finite field bound} can also be deduced from the methods of the paper \cite{MR2336319}.

\section{Measures with a spectral gap}\label{sec:gap}

In this section, $\psi \in \mathcal{S}(\R^{2})$ is a fixed non-negative Schwartz function with the properties
\begin{equation}\label{form46} \psi(0) = 1 \quad \text{and} \quad |\hat{\psi}(\xi) - \hat{\psi}(0)| \lesssim |\xi|^{2}. \end{equation}
For example, the $2$-dimensional Gaussian $\psi(x) = e^{-\pi |x|^{2}}$ satisfies \eqref{form46}. From \eqref{form46}, it follows that there exists a constant $c_{\psi} > 0$ such that 
\begin{equation}\label{form41} \psi(x) \geq \tfrac{1}{2}, \qquad |x| \leq c_{\psi}, \end{equation}
and the implicit constants in the "$\lesssim$" notation of this section will be allowed to depend on this constant "$c_{\psi}$". 

Whenever $\mu$ is a finite Borel measure on $\R^{2}$, we will in this section use the notation $\mu_\delta = \mu \ast \psi_\delta$, where $\psi_\delta(x) = \delta^{-2} \psi(x/\delta)$, $x \in \R^2$.

We recall that the \emph{$\sigma$-dimensional Riesz energy} of a Borel measure $\mu$ on $\R^{d}$ is
\begin{displaymath} I_{\sigma}(\mu) := \iint \frac{d\mu(x) d\mu (y)}{|x - y|^{\sigma}}, \qquad \sigma \geq 0. \end{displaymath}
We also recall, see for example \cite[Lemma 12.12]{zbMATH01249699}, that the Riesz energy of a compactly supported Borel measure $\mu$ can be expressed in terms of the Fourier transform as follows:
\begin{equation}\label{riesz} I_{\sigma}(\mu) = c(d,\sigma) \int |\hat{\mu}(\xi)|^{2}|\xi|^{\sigma - d} \, d\xi, \qquad 0 < \sigma < d, \end{equation}
where $c(d,\sigma) > 0$. We also record here the easy fact that if $\mu$ is a Borel measure on $\R^{d}$ satisfying $\|\mu\| \leq C$, and $\mu(B(x,r)) \leq Cr^{t}$ for some constants $C > 0, t \in [0,d]$, and for all $x \in \R^{d}$ and $r > 0$, then $I_{s}(\mu) \lesssim_{C,s,t} 1$ for all $0 \leq s < t$. For a proof, see \cite[Section 8]{zbMATH01249699}. 

 We then start the proof of Theorem \ref{mainThm} with the following lemma: it gives sufficient conditions for the integral in Proposition \ref{auxProp} to be positive.

\begin{lemma}\label{lemma1} Let $\mathbf{C} \geq 1$ and $3/2 < \sigma < 2$. Then there exist an absolute constant $A_{0} \geq 1$ such that the following holds for all $A \geq A_{0}$, and all $B \geq B_{0}(A,\mathbf{C},\sigma)$. Let $\mu$ be a Borel probability measure on $[0,1]^{2}$ such that $I_{\sigma}(\mu) \leq \mathbf{C}$. Assume additionally that
\begin{equation}\label{form22} \int_{\sqrt[5]{A} \leq |\xi| \leq B^{2}} |\hat{\mu}(\xi)|^{2} \, d\xi \leq A^{-2}   \end{equation}
Let $\pi$ be the length measure on the truncated parabola $\{(z,z^{2}) : A^{-2} \leq |z| \leq 1\}$. Then
\begin{equation}\label{form23} \int \mu_{\delta} \ast \pi \, d\mu \geq A^{-2} \end{equation}
for all sufficiently small $\delta > 0$. \end{lemma}
\begin{proof} Let $A_{0} \geq 1$ be an absolute constant to be selected later, at the very end of the proof, and let $A_{0} \leq A < B$. We will assume that 
\begin{equation}\label{rev3} B \geq B_{0}(A,\mathbf{C},\sigma) := (A\mathbf{C}C_{\sigma})^{5/(3/2 - \sigma)}, \end{equation}
where $C_{\sigma} \geq 1$ is a constant depending only on $\sigma \in (\tfrac{3}{2},2)$. To be accurate, $C_{\sigma}$ will also depend on the choice of "$\psi$". For $0 < \delta < B^{-1}$, write
\begin{displaymath} \int \mu_{\delta} \ast \pi \, d\mu = \int \mu_{1/A} \ast \pi \, d\mu + \int (\mu_{1/B} - \mu_{1/A}) \ast \pi \, d\mu + \int (\mu_{\delta} - \mu_{1/B}) \ast \pi \, d\mu =: I_{1} + I_{2} + I_{3}. \end{displaymath}
We first claim that
\begin{equation}\label{form43} I_{1} \geq \kappa \cdot A^{-1}, \end{equation}
where $\kappa > 0$ only depends on the choice of the fixed Schwartz function $\psi$. To see this, fix a small absolute constant $c > 0$, and assume temporarily that $x \in \spt (\mu)$ is a point with $\mu(B(x,r)) \geq cr^{2}$ for all $0 < r \leq 1$. Then, for $|y - x| \leq c_{\psi}/(2A)$, and using the non-negativity of $\psi$, we have
\begin{align*} \mu_{1/A}(y) = A^{2} \int \psi(A(y - z)) \, d\mu(z) & \stackrel{\eqref{form41}}{\geq} \tfrac{1}{2}A^{2} \mu(B(y,c_{\psi}/A))\\
&\,\, \geq\, \tfrac{1}{2}A^{2}\mu(B(x,c_{\psi}/(2A))) \gtrsim_{c,\psi} 1.  \end{align*}
Recalling that $\pi$ is the length measure on $\{(z,z^{2}) : A^{-2} \leq |z| \leq 1\}$, it follows that
\begin{displaymath} (\mu_{1/A} \ast \pi)(x) = \int \mu_{1/A}(x - y) \, d\pi(y) \gtrsim_{c,\psi} \pi(B(0,c_{\psi}/(2A))) \gtrsim_{\psi} A^{-1} \end{displaymath}
for every point $x \in \spt (\mu)$ with the uniform density lower bound $\mu(B(x,r)) \geq cr^{2}$. Now, to prove \eqref{form43}, it remains to note that if $c > 0$ is chosen sufficiently small, then these points have $\mu$ measure at least $\tfrac{1}{2}$. By the $5r$-covering theorem, the complement of these points in $\spt \mu$, say $\mathbf{Bad} \subset \spt \mu$, can be covered by discs $\{B(x_{j},5r_{j})\}_{j \in \N}$ such that $x_{j} \in \mathbf{Bad} \subset [0,1]^{2}$, $0 < 5r_{j} \leq 1$, the discs $B(x_{j},r_{j})$ are disjoint, and and $\mu(B(x,5r_{j})) \leq 25cr_{j}^{2}$ for all $j \geq 1$. Therefore,
\begin{displaymath} \mu(\mathbf{Bad}) \leq \sum_{j \in \N} \mu(B(x,5r_{j})) \leq 25c \sum_{j \in \N} r_{j}^{2} \sim c \sum_{j \in \N} \mathcal{L}^{2}(B(x_{j},r_{j})) \leq c\mathcal{L}^{2}(B(0,2)) \sim c. \end{displaymath} 
Now, if $c > 0$ is chosen sufficiently small, we see that $\mu(\mathbf{Bad}) \leq \tfrac{1}{2}$, as claimed. Therefore, \eqref{form43} follows from
\begin{displaymath} I_{1} \geq \int_{\R^{2} \, \setminus \mathbf{Bad}} \mu_{1/A} \ast \pi \, d\mu \gtrsim_{c,\psi} \mu(\R^{2} \, \setminus \, \mathbf{Bad}) \cdot A^{-1} \sim A^{-1}. \end{displaymath}
The terms $I_{2},I_{3}$ may be negative, but they are small compared to $1/A$. We start with $I_{2}$:
\begin{displaymath} I_{2} \leq \int |\hat{\mu}(\xi)|^{2}|\widehat{\psi}(\xi/B) - \widehat{\psi}(\xi/A)| |\hat{\pi}(\xi)| \, d\xi \lesssim \int |\widehat{\psi}(\xi/B) - \widehat{\psi}(\xi/A)||\hat{\mu}(\xi)|^{2} \, d\xi. \end{displaymath}
The integral over $\sqrt[5]{A} \leq |\xi| \leq B^{2}$ is bounded by $\lesssim A^{-2}$ by the spectral gap assumption \eqref{form22}. Using also that $\|\hat{\mu}\|_{L^{\infty}} \leq 1$, it suffices to find upper bounds for
\begin{displaymath} I_{2}^{'} := \int_{|\xi| \leq \sqrt[5]{A}} |\widehat{\psi}(\xi/B) - \widehat{\psi}(\xi/A)| \, d\xi \quad \text{and} \quad I_{2}^{''} := \int_{|\xi| \geq B^{2}} |\widehat{\psi}(\xi/B) - \widehat{\psi}(\xi/A)| \, d\xi. \end{displaymath}
Regarding $I_{2}'$, we simply estimate the integrand with the triangle inequality, and then plug in the estimate \eqref{form46}:
\begin{align*} |\widehat{\psi}(\xi/B) - \widehat{\psi}(\xi/A)| & \leq |\widehat{\psi}(\xi/B) - \widehat{\psi}(0)| + |\widehat{\psi}(\xi/A) - \widehat{\psi}(0)| \lesssim \left|\frac{\xi}{B}\right|^{2}  + \left|\frac{\xi}{A}\right|^{2} \lesssim A^{2/5 - 2}. \end{align*} 
Consequently,
\begin{equation}\label{rev1} I_{2}' \lesssim \int_{|\xi| \leq \sqrt[5]{A}} A^{2/5 - 2} \, d\xi \lesssim A^{-1.2}. \end{equation}
We turn to estimating $I_{2}''$, where usual Schwartz function bounds for $\widehat{\psi}$ are sufficient:
\begin{displaymath} I_{2}'' \leq \int_{|\xi| \geq B^{2}} |\widehat{\psi}(\xi/B)| \, d\xi + \int_{|\xi| \geq B^{2}} |\widehat{\psi}(\xi/A)| \, d\xi \lesssim B^{2}\int_{|\xi| \geq B} |\widehat{\psi}(\xi)| \, d\xi \lesssim_{\psi} B^{-2} \leq A^{-2}. \end{displaymath} 
We finally estimate term $I_{3}$:
\begin{displaymath} I_{3} = \int |\hat{\mu}(\xi)|^{2}|\widehat{\psi}(\xi/B) - \widehat{\psi}(\delta \xi)|^{2}|\hat{\pi}(\xi)| \, d\xi \lesssim I_{3}' + I_{3}'', \end{displaymath}
where 
\begin{equation}\label{rev2} I_{3}' := \int_{|\xi| \leq \sqrt[5]{B}} |\widehat{\psi}(\xi/B) - \widehat{\psi}(\delta \xi)| \, d\xi \lesssim B^{-1.2} \leq A^{-1.2} \end{equation}
by the same argument we used to estimate $I_{2}'$ (recall that $\delta < B^{-1}$), and
\begin{displaymath} I_{3}'' := \int_{|\xi| \geq \sqrt[5]{B}} |\hat{\mu}(\xi)|^{2}|\hat{\pi}(\xi)|^{2} \, d\xi. \end{displaymath}
Using $|\hat{\pi}(\xi)| \lesssim |\xi|^{-1/2}$, and $I_{\sigma}(\mu) \leq \mathbf{C}$, we find
\begin{align*} I_{3}'' & \lesssim \int_{|\xi| \geq \sqrt[5]{B}} |\hat{\mu}(\xi)|^{2}|\xi|^{-1/2} \, d\xi\\
& \leq (\sqrt[5]{B})^{3/2 - \sigma} \int_{|\xi| \geq \sqrt[5]{B}} |\hat{\mu}(\xi)|^{2}|\xi|^{\sigma - 2} \, d\xi \stackrel{\eqref{riesz}}{\lesssim_{\sigma}} (\sqrt[5]{B})^{3/2 - \sigma} \mathbf{C}, \end{align*}
Recalling the constant $\kappa = \kappa(\psi) > 0$ from \eqref{form43}, we therefore have $I_{3}'' \leq (\kappa/2) \cdot A^{-1}$ for $B \geq (A\mathbf{C}C_{\sigma})^{5/(3/2 - \sigma)}$, if $C_{\sigma} \geq 1$ is chosen sufficiently large in terms of $\sigma$ and $\kappa = \kappa(\psi)$. This is what we assumed in \eqref{rev3}. Finally, combining \eqref{form43}, \eqref{rev1}, and \eqref{rev2}, we find that
\begin{displaymath} \int \mu_{\delta} \ast \pi \, d\mu \geq (\kappa/2) \cdot A^{-1} - O(A^{-1.2}) - O(A^{-1.2}). \end{displaymath}
In particular the integral is $\geq A^{-2}$ for $A \geq A_{0}$ for a sufficiently large absolute constant $A_{0} \geq 1$; this is finally the place where the constant $A_{0}$ is chosen. \end{proof}

\section{Finding measures with a spectral gap}\label{s:construction}

In view of Proposition \ref{auxProp} and Lemma \ref{lemma1}, the first attempt to prove Theorem \ref{mainThm} might be to introduce a Frostman measure $\nu$ with $\spt (\nu) \subset K$, and $I_{\sigma}(\nu) < \infty$ for some $\sigma$ close to $\Hd K > 3/2$. The issue is that such a measure may not satisfy the "spectral gap" property \eqref{form22}. Locating a measure with property \eqref{form22} is the main challenge of the proof.

We introduce a few pieces of notation. The metric space $(\R^{2},d_{\Pi})$,
\begin{displaymath} d_{\Pi}((x,s),(y,t)) := \max\{|x - y|,|s - t|^{1/2}\}, \qquad (x,s),(y,t) \in \R^{2}, \end{displaymath}
is called the \emph{parabolic plane}. The parabolic plane admits a natural collection of \emph{dyadic rectangles} $\mathcal{D}$, which we now describe. We declare that $\mathcal{D}_{0} := \{x + [0,1)^{2} : x \in \Z^{2}\}$. For $j \geq 0$, rectangles in $\mathcal{D}_{j + 1}$ are obtained by sub-dividing each rectangle in $\mathcal{D}_{j}$ into two equal parts vertically, and four equal parts horizontally. Thus, for example 
\begin{displaymath} Q_{j} := [0,2^{-j}) \times [0,4^{-j}) \in \mathcal{D}_{j}, \end{displaymath}
and all elements of $\mathcal{D}_{j}$ are translates of $Q_{j}$. Finally, we write $\mathcal{D} := \bigcup_{j \geq 0} \mathcal{D}_{j}$. For $Q \in \mathcal{D}_{j} \subset \mathcal{D}$, we define 
\begin{displaymath} \ell(Q) := 2^{-j} = \diam_{\Pi}(Q). \end{displaymath}
For each rectangle $Q = [x,x + 2^{-j}) \times [s,s + 4^{-j})$, we define the \emph{rescaling map} $T_{Q} \colon Q \to [0,1)^{2}$, as follows:
\begin{equation}\label{TQ} T_{Q}(y,t) := (2^{j}(y - x),4^{j}(t - s)), \qquad (y,t) \in \R^{2}. \end{equation} 
Thus $T_{Q}(Q) = [0,1)^{2} \in \mathcal{D}_{0}$. If $\mu$ is a Borel measure on $\R^{2}$, and $\mu(Q) > 0$ for some $Q \in \mathcal{D}$, we defined the \emph{renormalised blow-up} 
\begin{displaymath} \mu^{Q} := \|\mu|_{Q}\|^{-1} \cdot T_{Q}(\mu|_{Q}). \end{displaymath}
Then $\mu^{Q}$ is a Borel probability measure with $\spt (\mu^{Q}) \subset [0,1]^{2}$. We will need a version of Frostman's lemma in the parabolic plane; it could be deduced from Frostman's lemma in metric spaces, but it is much easier to use the dyadic rectangles $\mathcal{D}$ in $(\R^{2},d_{\Pi})$, and follow the argument from Euclidean spaces, see \cite[Theorem 8.8]{zbMATH01249699}.
\begin{lemma}[Frostman's lemma]\label{frostman} Let $K \subset (\R^{2},d_{\Pi})$ be a compact set with $\mathcal{H}^{s}_{\infty}(K) =: \tau$. Here $\mathcal{H}^{s}_{\infty}$ refers to Hausdorff content defined relative to the metric $d_{\Pi}$. Then, there exists an absolute constant $C > 0$, and a Borel measure $\mu$ supported on $K$ with the properties
\begin{itemize}
\item $\|\mu\| \geq \tau$, and
\item $\mu(B_{\Pi}(x,r)) \leq Cr^{s}$ for all $x \in \R^{2}$ and $r > 0$.
\end{itemize}
\end{lemma}
A measure $\mu$ satisfying the conclusions of Lemma \ref{frostman} is called a \emph{parabolic $s$-Frostman measure}. We record that every parabolic $s$-Frostman measure $\mu$ is a Euclidean $(s - 1)$-Frostman measure. Indeed, for $r > 0$, every Euclidean disc $B(x,r)$ can be covered by $\leq 5r^{-1}$ parabolic balls $B_{\Pi}(x_{j},r)$. Hence 
\begin{displaymath} \mu(B(x,r)) \leq 5r^{-1} \max_{j} \mu(B_{\Pi}(x_{j},r)) \leq 5Cr^{s - 1}. \end{displaymath}

We are now ready to state the key proposition:

\begin{proposition}\label{mainProp} There exists an absolute constant $\epsilon > 0$ such that the following holds. Let $K \subset \R^{2}$ be a compact set with parabolic Hausdorff dimension $\Hd^{\Pi} K > 3 - \epsilon$. Then, there exists a parabolic rectangle $\mathbf{Q} \in \mathcal{D}$ and a non-trivial measure $\mu_{\mathbf{Q}}$ with $\spt \mu_{\mathbf{Q}} \subset K \cap \mathbf{Q}$, such that the renormalised blow-up $\mu = (\mu_{\mathbf{Q}})^{\mathbf{Q}}$ satisfies the hypotheses of Lemma \ref{lemma1} for suitable absolute constants $A,B,\mathbf{C} > 0$, and $\sigma = 10/6 > 3/2$.
\end{proposition}

\begin{proof} We may assume with no loss of generality that $K \subset [0,1)^{2}$, since $d_{\Pi}$ is translation invariant. We begin by fixing certain constants $A,B,\mathbf{C} \geq 1$ and $T \in \N$. All of these constants will be absolute, and here is how they need to be chosen. The constant $A \geq 1$ is chosen in such a manner that
\begin{equation}\label{form32} \int_{|\xi| \geq \sqrt[5]{A}} |\widehat{\varphi}(\xi)| \, d\xi \leq A^{-3} \end{equation}
for some non-negative function $\varphi \in C^{\infty}(\R^{2})$ with $\spt \varphi \subset (0,1)^{2}$ and $\int \varphi(x) \, dx = 1$. Any choice of "$\varphi$" with these properties will work for us, and we may for example choose $\varphi$ in such a way that
\begin{equation}\label{form42} \|\varphi\|_{L^{\infty}} \leq 2. \end{equation}
Since the Fourier transform of any such function $\varphi$ is a Schwartz function, $A \geq 1$ exists. We will also require $A \geq 1$ to be so large that $CA^{-3} \leq A^{-2}$ for a suitable absolute constant $C \geq 1$, appearing later. Next, the constant $\mathbf{C}$ is chosen in such a way that if $\nu$ is a Borel probability measure on $[0,1]^{2}$ with $\nu(B(x,r)) \leq Cr^{11/6}$ for all $x \in \R^{2}$, all $r \in (0,1]$, and a suitable absolute constant $C \geq 1$ to be determined later, then 
\begin{displaymath} I_{10/6}(\nu) \leq \mathbf{C}. \end{displaymath}
This is possible by the remark below \eqref{riesz}. We also note that $10/6 > 3/2$. Next, the constant $B = B_{0}(A,\mathbf{C},10/6)$ is chosen as in Lemma \ref{lemma1}, given the parameters $A,\mathbf{C},\sigma = 10/6$. Finally, the parameter $T \in \N$ is chosen so large that
\begin{equation}\label{form31} 2^{-T}B^{6} \leq A^{-3}. \end{equation}
We emphasise that the "$\lesssim$" notation below only hides absolute constants, and the implicit constants will not depend on $A,B,T,\mathbf{C}$. We then begin the proof in earnest. Let $s := \Hd^{\Pi} K > 3 - \epsilon$. We assume with no loss of generality that $s > 17/6$, and $\mathcal{H}^{s}(K) > 0$. Here $\mathcal{H}^{s}$ refers to parabolic Hausdorff measure. We will also use the following \emph{dyadic parabolic Hausdorff content}:
\begin{displaymath} \mathcal{H}^{s}_{\infty}(K) := \inf\left\{\sum_{i} \ell(Q_{i})^{s} : K \subset \bigcup_{i} Q_{i}, \, Q_{i} \in \mathcal{D} \right\}. \end{displaymath}
Evidently $\mathcal{H}^{s}(K) > 0$ implies $\mathbf{K} := \mathcal{H}^{s}_{\infty}(K) > 0$, and on the other hand 
\begin{displaymath} \mathbf{K} \leq \ell([0,1]^{2})^{s} = 1. \end{displaymath}
We claim that for every $\delta > 0$ there exists a rectangle $\mathbf{Q} \in \mathcal{D}$ such that
\begin{equation}\label{form27} \mathcal{H}^{s}_{\infty}(K \cap \mathbf{Q}) \geq (1 - \delta) \ell(\mathbf{Q})^{s}. \end{equation}
Indeed, for every $\tau > 0$, we may cover $K$ with a family of rectangles $Q_{1},Q_{2},\ldots \in \mathcal{D}$ such that $\sum_{i} \ell(Q_{i})^{s} \leq \mathbf{K} + \tau$. If all of these rectangles failed to satisfy \eqref{form27}, then
\begin{displaymath} \mathbf{K} = \mathcal{H}^{s}_{\infty}(K) \leq \sum_{i} \mathcal{H}^{s}_{\infty}(K \cap Q_{i}) \leq (1 - \delta) \sum_{i} \ell(Q_{i})^{s} \leq (1 - \delta) (\mathbf{K} + \tau). \end{displaymath}
Since $\mathbf{K} > 0$, this inequality is a contradiction for $\tau > 0$ small enough, and for all such $\tau > 0$ one of the rectangles $\mathbf{Q} := Q_{i}$ satisfies \eqref{form27}.  

We now fix $\delta := \tfrac{1}{8} \cdot 2^{-3T}$, where $T \in \N$ was chosen at \eqref{form31}, and we pick $\mathbf{Q} \in \mathcal{D}$ satisfying \eqref{form27} for this $\delta$. Let $\mathrm{ch}(\mathbf{Q}) \subset \mathcal{D}$ be the generation-$T$ children of $\mathbf{Q}$. We claim that if $\epsilon = 3 - s$ is small enough (depending on $T$), then
\begin{equation}\label{form28} \mathcal{H}^{s}_{\infty}(K \cap Q) \geq \tfrac{1}{2} \cdot \ell(Q)^{s}, \qquad Q \in \mathrm{ch}(\mathbf{Q}). \end{equation}
To see this, let $\mathcal{G} := \{Q \in \mathrm{ch}(\mathbf{Q}) : \mathcal{H}^{s}_{\infty}(K \cap Q) \geq \tfrac{1}{2} \cdot \ell(Q)^{s} \}$. If $\mathcal{G}$ is a strict subset of $\mathrm{ch}(\mathbf{Q})$, then we use the inequality $\mathcal{H}_{\infty}^{s}(K \cap Q) \leq \ell(Q)^{s}$ to deduce the following estimate:
\begin{align*} 1 - \tfrac{1}{8} \cdot 2^{- 3T} \stackrel{\eqref{form27}}{\leq} \frac{\mathcal{H}^{s}_{\infty}(K \cap \mathbf{Q})}{\ell(\mathbf{Q})^{s}} & \leq \sum_{Q \in \mathcal{G}} \frac{\ell(Q)^{s}}{\ell(\mathbf{Q})^{s}} + (1 - \tfrac{1}{2}) \sum_{Q \in \mathrm{ch}(\mathbf{Q}) \, \setminus \, \mathcal{G}} \frac{\ell(Q)^{s}}{\ell(\mathbf{Q})^{s}}\\
& \leq \sum_{Q \in \mathrm{ch}(\mathbf{Q})} \frac{\ell(Q)^{s}}{\ell(\mathbf{Q})^{s}} - \tfrac{1}{2} \cdot 2^{-sT} \leq \left(\frac{\ell(Q)}{\ell(\mathbf{Q})} \right)^{s - 3} - \tfrac{1}{2} \cdot 2^{-3T}. \end{align*}
Note that the right hand side tends to $1 - \tfrac{1}{2} \cdot 2^{-3T}$ as $s \to 3$. More precisely, since $\ell(Q)/\ell(\mathbf{Q}) = 2^{-T}$, there exists $\epsilon = \epsilon(T) > 0$ such that the right hand side is $\leq 1 - \tfrac{1}{4} \cdot 2^{- 3T}$ for $s > 3 - \epsilon$. For such $\epsilon = 3 - s > 0$, the inequality above produces a contradiction. We assume that $\epsilon > 0$ is sufficiently small in the sequel, and hence \eqref{form28} holds. In particular, for every $Q \in \mathrm{ch}(\mathbf{Q})$ we may use the parabolic Frostman lemma, Lemma \ref{frostman}, to construct a measure $\mu_{Q}^0$ with the following properties:
\begin{itemize}
\item $\spt \mu_{Q}^0 \subset K \cap \overline{Q}$,
\item $\|\mu_{Q}^0\| \geq \tfrac{1}{2} \cdot \ell(Q)^{s}$, 
\item $\mu_{Q}^0(B_{\Pi}(x,r)) \leq Cr^{s}$ for all $x \in \R^{2}$ and $r > 0$.
\end{itemize}
Here $C > 0$ is an absolute constant. 
Recall the function $\varphi \in C_{c}^{\infty}(\R^{2})$ from \eqref{form32}, and define the measures $\mu_Q = \frac{w(Q)\ell(\mathbf{Q})^s}{\|\mu_Q^0\|} \mu_Q^0$ for $Q \in \mathrm{ch}(\mathbf{Q})$ using the weights
\begin{equation}\label{wQ} w(Q) := \int_{T_{\mathbf{Q}}(Q)} \varphi(x) \, dx, \qquad Q \in \mathrm{ch}(\mathbf{Q}). \end{equation}
Here $T_{\mathbf{Q}} \colon \mathbf{Q} \to [0,1)^{2}$ is the map defined in \eqref{TQ}. Recalling from \eqref{form42} that $\|\varphi\|_{L^{\infty}} \leq 2$, we have $w(Q) \leq 2\mathcal{L}^{2}(T_{\mathbf{Q}}(Q)) = 2 \cdot 2^{-3T} = 2 \cdot (\ell(Q)/\ell(\mathbf{Q}))^{3}$. Therefore (since $s \leq 3$),
\begin{displaymath}
w(Q)\ell(\mathbf{Q})^{s} \leq 2 \cdot \ell(Q)^{s} \leq 4\|\mu_{Q}^0\|, 
\end{displaymath}
from which it follows that $\mu_{Q}(B_{\Pi}(x,r)) \leq 4Cr^{s}$. We further observe that 
\begin{equation}\label{form36} \|\mu_{Q}\| = w(Q)\ell(\mathbf{Q})^{s} \leq (2\cdot 2^{-3T})\ell(\mathbf{Q})^{s}, \qquad Q \in \mathrm{ch}(\mathbf{Q}). \end{equation}
We then define
\begin{displaymath} \mu_{\mathbf{Q}} := \sum_{Q \in \mathrm{ch}(\mathbf{Q})} \mu_{Q}. \end{displaymath}
Thus, 
\begin{equation}\label{form37} \|\mu_{\mathbf{Q}}\| = \ell(\mathbf{Q})^{s} \sum_{Q \in \mathrm{ch}(\mathbf{Q})} w(Q) = \ell(\mathbf{Q})^{s} \sum_{Q \in \mathrm{ch}(\mathbf{Q})} \int_{T_{\mathbf{Q}}(Q)} \varphi(x) \, dx = \ell(\mathbf{Q})^{s}. \end{equation}
Moreover, $\mu_{\mathbf{Q}}$ is also a parabolic $s$-Frostman measure. Indeed, the Frostman estimate $\mu_{\mathbf{Q}}(B_{\Pi}(x,r)) \lesssim r^{s}$ is clear for $0 < r \leq 2^{-T}\ell(\mathbf{Q})$, since then $B_{\Pi}(x,r)$ only intersects a bounded number of the supports of the measures $\mu_{Q}$. The estimate is also clear for $r \geq \ell(\mathbf{Q})$ by \eqref{form37}. Finally, if $2^{-T}\ell(\mathbf{Q}) \leq r \leq \ell(\mathbf{Q})$, then (using again that $s \leq 3$),

\begin{displaymath} \mu_{\mathbf{Q}}(B_{\Pi}(x,r)) \leq \mathop{\sum_{Q \in \mathrm{ch}(\mathbf{Q})}}_{Q \cap B_{\Pi}(x,r) \neq \emptyset} \|\mu_{Q}\| \stackrel{\eqref{form36}}{\lesssim} (r/(2^{-T}\ell(\mathbf{Q})))^{3} \cdot 2^{-3T} \ell(\mathbf{Q})^{s} \leq r^{s}. \end{displaymath}

Next, we let $\mu := \ell(\mathbf{Q})^{-s}T_{\mathbf{Q}}\mu_{\mathbf{Q}} = \|\mu_{\mathbf{Q}}\|^{-1}T_{\mathbf{Q}}\mu_{\mathbf{Q}}$ be the renormalised blow-up of $\mu_{\mathbf{Q}}$. Then $\mu$ is a Borel probability measure with $\spt (\mu) \subset [0,1]^{2}$, and $\mu$ is also a parabolic $s$-Frostman measure (this crucially uses the fact that the renormalisation constant $\|\mu_{\mathbf{Q}}\|^{-1}$ coincides with $\ell(\mathbf{Q})^{-s}$). Since we assumed that $s > 17/6$, we may now infer, in particular, that $\mu$ is a Euclidean $(11/6)$-Frostman measure, and we have computed above that the Frostman constant of $\mu$ is bounded by some absolute constant $C > 0$. It follows that $I_{10/6}(\mu) \leq \mathbf{C}$, assuming that $\mathbf{C}$ was initially chosen large enough depending on this $C$.

To conclude the proof, it only remains to verify that $\mu$ satisfies the "spectral gap" condition \eqref{form22}. Let 
\begin{displaymath} \mathrm{ch}(Q_{0}) := \{Q \in \mathcal{D} : \ell(Q) = 2^{-T}\} = \{T_{\mathbf{Q}}(Q) : Q \in \mathrm{ch}(\mathbf{Q})\}. \end{displaymath}
Then
\begin{equation}\label{form29} \mu(Q) = \frac{\mu_{\mathbf{Q}}(T_{\mathbf{Q}}^{-1}(Q))}{\ell(\mathbf{Q})^{s}} \stackrel{\eqref{form36}}{=} w(T_{\mathbf{Q}}^{-1}(Q)) \stackrel{\eqref{wQ}}{=} \int_{Q} \varphi(x) \, dx, \qquad Q \in \mathrm{ch}(Q_{0}), \end{equation}
We will also treat $\varphi$ as a measure, so we abbreviate $\int f \, d\varphi := \int f(x)\varphi(x) \, dx$. We will now estimate the Fourier transform of $\mu$ by comparing it to the Fourier transform of $\varphi$. Since $\mu(Q) = \varphi(Q)$ for all $Q \in \mathrm{ch}(Q_{0})$, we have also
\begin{displaymath} \int_{Q} e^{-2\pi c_{Q} \cdot \xi} \, d\mu(x) = \int_{Q} e^{-2\pi i c_{Q} \cdot \xi} d\varphi(x), \qquad \xi \in \R^{2}, \, Q \in \mathrm{ch}(Q_{0}),  \end{displaymath}
where $c_{Q} \in Q$ refers to the midpoint of the rectangle $Q$. For $\xi \in \R^{2}$ fixed, noting that $x \mapsto e^{-2\pi i x \cdot \xi}$ is $C|\xi|$-Lipschitz, and $|x - c_{Q}| \leq 2^{-T}$ for $x \in Q \in \mathrm{ch}(Q_{0})$, this allows us to estimate as follows:
\begin{align*} |\hat{\mu}(\xi) - \widehat{\varphi}(\xi)| & = \left| \int e^{-2\pi i x \cdot \xi} \, d\mu(x) - \int e^{-2\pi i x \cdot \xi} \, d\varphi(x) \right|\\
& \leq \sum_{Q \in \mathrm{ch}(Q_{0})} \int_{Q} |e^{-2\pi i x \cdot \xi} - e^{-2\pi i c_{Q} \cdot \xi}| \, d\|\mu - \varphi\|(x)\\
& \lesssim |\xi| \cdot 2^{-T} \sum_{Q \in \mathrm{ch}(Q_{0})} \|\mu - \varphi\|(Q) \lesssim |\xi| \cdot 2^{-T}. \end{align*}
Consequently, noting also that $\|\hat{\mu}\|_{L^{\infty}} \leq 1$ and $\|\widehat{\varphi}\|_{L^{\infty}} \leq 1$, we have
\begin{align} \int_{\sqrt[5]{A} \leq |\xi| \leq B^{2}} |\hat{\mu}(\xi)|^{2} \, d\xi & \lesssim \int_{\sqrt[5]{A} \leq |\xi| \leq B^{2}} |\hat{\mu}(\xi) - \widehat{\varphi}(\xi)| \, d\xi + \int_{|\xi| \geq \sqrt[5]{A}} |\widehat{\varphi}(\xi)| \, d\xi \notag\\
&\label{form44} \stackrel{\eqref{form32}}{\lesssim} 2^{-T} \int_{|\xi| \leq B^{2}} |\xi| \, d\xi + A^{-3} \lesssim 2^{-T}B^{6} + A^{-3}.  \end{align} 
Since $2^{-T}B^{6} \leq A^{-3}$ by \eqref{form31}, we finally see that 
\begin{displaymath} \int_{\sqrt[5]{A} \leq |\xi| \leq B^{2}} |\hat{\mu}(\xi)|^{2} \, d\xi \leq CA^{-3} \leq A^{-2}, \end{displaymath}
assuming that "$A$" was chosen large enough to begin with, depending only only the absolute constants implicit in \eqref{form44}. This completes the proof of Proposition \ref{mainProp}. \end{proof}

\section{Concluding the proof of Theorem \ref{mainThm}}\label{sec:conclusion}

Theorem \ref{mainThm} follows easily by combining Lemma \ref{lemma1} and Propositions \ref{auxProp}-\ref{mainProp}. 

\begin{proof}[Proof of Theorem \ref{mainThm}] The following relation holds between Euclidean and parabolic Hausdorff dimension:
\begin{displaymath} \Hd^{\Pi} K \geq 2\Hd K - 1. \end{displaymath}
This is easiest to see by observing that every Euclidean $s$-Frostman measure is a parabolic $(2s - 1)$-Frostman measure: for $r \in (0,1]$, a parabolic ball $B_{\Pi}(x,r)$ can be covered by $\sim r^{-1}$ Euclidean discs of radius $r^{2}$. Therefore, if $\epsilon > 0$ is the absolute constant from Proposition \ref{mainProp}, and $\Hd K > 2 - \epsilon/2$, then $\Hd^{\Pi} K > 3 - \epsilon$, and Proposition \ref{mainProp} becomes applicable.

By Proposition \ref{mainProp}, there exists a parabolic rectangle $\mathbf{Q} \in \mathcal{D}$ and a non-trivial measure $\mu_{\mathbf{Q}}$ with $\spt (\mu_{\mathbf{Q}}) \subset K \cap \mathbf{Q}$ such that the renormalised blow-up $\mu := (\mu_{\mathbf{Q}})^{\mathbf{Q}}$ satisfies the hypotheses of Lemma \ref{lemma1} with constants $A,B,\mathbf{C}$ and $\sigma = 10/6$. Therefore
\begin{displaymath} \liminf_{\delta \to 0} \int (\mu \ast \psi_{\delta}) \ast \pi \, d\mu > 0, \end{displaymath}
where $\pi$ is the length measure on the truncated parabola $\{(z,z^{2}) : A^{-2} \leq |z| \leq 1\}$, and $\psi \in \mathcal{S}(\R^{2})$ is the Schwartz function defined in \eqref{form46}. It now follows from Proposition \ref{auxProp} that there exist $(x_{1},t_{1}) \in \spt (\mu)$ and $z \neq 0$ such that $(x_{1} + z,t_{1} + z^{2}) \in \spt(\mu)$. 

Next, note that $\spt (\mu) = T_{\mathbf{Q}}(\spt (\mu_{\mathbf{Q}}))$, so 
\begin{equation}\label{form40} T^{-1}_{\mathbf{Q}}(\{(x_{1},t_{1}),(x_{1} + z,t_{1} + z^{2})\}) \subset T^{-1}_{\mathbf{Q}}(\spt (\mu)) = \spt (\mu_{\mathbf{Q}}) \subset K. \end{equation}
Since $T_{\mathbf{Q}}(x,t) = (2^{j}(x - x_{0}),4^{j}(t - t_{0}))$ for some $j \geq 0$ and $(x_{0},t_{0}) \in \R^{2}$, the inverse $T^{-1}_{\mathbf{Q}}$ has the formula $T^{-1}_{\mathbf{Q}}(x,t) = (2^{-j}x + x_{0}, 4^{-j}t + t_{0})$. Therefore, \eqref{form40} is equivalent to
\begin{displaymath} \mathbf{x} := (2^{-j}x_{1} + x_{0},4^{-j}t_{1} + t_{0}) \in K \quad \text{and} \quad (2^{-j}(x_{1} + z) + x_{0},4^{-j}(t_{1} + z^{2}) + t_{0}) \in K. \end{displaymath}
Finally, note that the point on the right equals $\mathbf{x} + (\bar{z},\bar{z}^{2})$ with $\bar{z} := 2^{-j}z \neq 0$. Therefore $\{\mathbf{x},\mathbf{x} + (\bar{z},\bar{z}^{2})\} \subset K$, and the proof is complete. \end{proof}

\section{The proof of Theorem \ref{finite field bound}}\label{section on finite fields}
Let $\FF_q$ be a finite field with $q$ elements, where $q = p^n$ is a power of a prime $p>2$. We conclude with the proof of Theorem \ref{finite field bound}, which is a consequence of the following counting result.
\begin{proposition}\label{count of progressions}
Let $A\subset\FF_q^2$ have cardinality $|A| = \alpha q^2$ for some $0\leq \alpha \leq 1$. Then
\begin{align}\label{our progression}
    \left\{(x,  y,  z)\in\FF_q^3: (x,  y),  (x + z,  y + z^2)\in A\}\right|\geq \left(\alpha - q^{-\frac{1}{2}}\right)\alpha q^3
\end{align}
\end{proposition}
Proposition \ref{count of progressions} is a consequence of the result below with $f = g = 1_A$, where $1_A$ is the indicator function of the set $A$, and the simple observation that $\|1_A\|_2 = \sqrt{\alpha}$ for the normalization of the $L^2$ norm given below in \eqref{L^p norm}. 
\begin{proposition}\label{equality of counting operators}
Let $f, g: \FF_q^2 \to \CC$. Then
\begin{align}\label{equality of counting operators - equation}
    \left|\sum_{x, y, z\in\FF_q} f(x, y) g(x + z, y+z^2) - \frac{1}{q}\sum_{x,y\in\FF_q} f(x,y) \sum_{x,y\in\FF_q} g(x,y) \right|\leq q^\frac{5}{2}\|f\|_2 \|g\|_2.
\end{align}
\end{proposition}

\begin{proof}[Proof of Theorem \ref{finite field bound}]
Let $A\subset \FF_q^2$ with $|A| = \alpha q^2$. We split the count of progressions $\{(x,  y),  (x + z,  y + z^2)\}$ inside $A$ into ``trivial'' progressions with $z = 0$ and nontrivial ones with $z\neq 0$:
\begin{equation}\label{form47}
    \sum_{x, y, z\in \FF_q}1_A(x, y) 1_A(x + z, y + z^2) = \sum_{x, y\in\FF_q}1_A(x, y) + \sum_{\substack{x, y\in\FF_q,\\ z\in \FF_q \, \setminus \, {\{0\}}}}1_A(x, y) 1_A(x + z, y + z^2).
\end{equation}
The first term on the right-hand side is just $|A| = \alpha q^2$. By Proposition \ref{count of progressions} applied to the left hand side of \eqref{form47}, the second term on the right-hand side of \eqref{form47} can be bounded from below as
\begin{align*}
    \sum_{\substack{x, y\in\FF_q,\\ z\in \FF_q \, \setminus \, {\{0\}}}}1_A(x, y) 1_A(x + z, y + z^2) \geq (\alpha -  q^{-\frac{1}{2}} - q^{-1}) \alpha q^3 > (\alpha - 2q^{-\frac{1}{2}}) \alpha q^3.
\end{align*}
The quantity $\alpha - 2q^{-\frac{1}{2}}$ is nonnegative whenever $\alpha \geq 2 q^{-\frac{1}{2}}$, and so $|A|$ contains a nontrivial progression whenever $|A| \geq 2 q^\frac{3}{2}$, as claimed. \end{proof}

The proof of Proposition \ref{equality of counting operators} relies on the Fourier decay of the indicator function of the parabola $\Pi = \{(z,z^2)\in\FF_q^2: z\in\FF_q\}$. To state this formally, we need to avail ourselves of the Fourier analysis on $\FF_q$. We introduce its basics now, referring the reader to e.g. \cite{Wan} for more details.

The Fourier transform of $f:\FF_q^2\to\CC$ is defined by the formula 
\begin{align*}
    \hat{f}(\xi_1, \xi_2) = \frac{1}{q^2} \sum_{x,y\in{\FF_q}}f(x, y) \xi_1(x) \xi_2(y),
\end{align*}
where $\xi_1, \xi_2:\FF_q\to\CC$ are characters, i.e. group homomorphisms from $\FF_q$ to the unit circle in $\CC$. Each character $\xi:\FF_q\to\CC$ takes the form $\xi(x) = e^{2\pi i \textrm{Tr}(a x)/p}$ for some $a\in\FF_q$, where $\textrm{Tr}(x) = x + x^p + \ldots + x^{p^{n-1}}$ for $x \in \FF_q$, recall $q = p^n$ and $p>2$ is prime. We rely on two main properties of characters: that characters are group homomorphisms, i.e. $\xi(x + y) = \xi(x)\xi(y)$ for every $x, y\in\FF_q$, and that they satisfy an orthogonality relation
\begin{align}\label{orthogonality}
    \sum_{x\in\FF_q}\xi(x) = \begin{cases} q,\; \xi = 1\\
    0,\; \xi\neq 1.
    \end{cases}
\end{align}

We denote the groups of characters on $\FF_q^2$ and $\FF_q$ as $\widehat{\FF_q^2}$ and $\widehat{\FF_q}$ respectively. The group $\widehat{\FF_q^2}$ is of the form $\widehat{\FF_q^2} = \widehat{\FF_q}\times \widehat{\FF_q}$, meaning that every character on $\FF_q^2$ takes the form $\xi(x, y) = \xi_1(x)\xi_2(y)$ for some characters $\xi_1, \xi_2:\FF_q\to\CC$. By Pontryagin duality, the group $\widehat{\FF_q^2}$ is isomorphic to $\FF_q^2$, and more explicitly, the isomorphism takes the form $$a\mapsto (x\mapsto e^{2\pi i \textrm{Tr}(a x)/p}).$$ The Fourier inversion formula in this case is
\begin{align*}
    f(x, y) = \sum_{\xi_1, \xi_2\in\widehat{\FF_q}}\hat{f}(\xi_1, \xi_2) \overline{\xi_1(x)\xi_2(y)}.
\end{align*}

Making the choices of normalization that have become standard in additive combinatorics, we define the $L^2$ norm of functions $f:\FF_q^2\to\CC$ and $\widehat{f}:\widehat{\FF_q^2}\to\CC$ as
\begin{align}\label{L^p norm}
    \|f\|_2 = \Big(\frac{1}{q^2}\sum\limits_{x, y\in\FF_q}|f(x,y)|^2\Big)^\frac{1}{2}\quad \textrm{and} \quad
    \|\widehat{f}\|_2 = \Big(\sum\limits_{\xi_1, \xi_2 \in \widehat{\FF_q}}\left|\widehat{f}(\xi_1, \xi_2)\right|^2\Big)^\frac{1}{2}.
    \end{align}
With these normalisations, the Parseval identity takes the simple form $\|f\|_2 = \|\hat{f}\|_2$.

\begin{proof}[Proof of Proposition \ref{equality of counting operators}]
Let $\Pi = \{(z,z^2)\in\FF_q^2: z\in\FF_q\}$ be the parabola inside $\FF_q^2$. We start by computing the values of its Fourier transform:
\begin{align*}
    \widehat{1_\Pi}(\xi_1, \xi_2) = \frac{1}{q^2}\sum_{x,y\in\FF_q}1_\Pi(x,y) \xi_1(x) \xi_2(y) = \frac{1}{q^2}\sum_{x\in\FF_q}\xi_1(x) \xi_2(x^2).
\end{align*}
Since we have assumed that the field $\FF_q$ has characteristic $p>2$, the sum $\sum\limits_{x\in\FF_q}\xi_1(x) \xi_2(x^2)$ equals $q$ when $\xi_1 = \xi_2 = 1$ are the trivial characters, otherwise its modulus is bounded from above by $q^\frac{1}{2}$. This is because when $\xi_2 = 1, \xi_1 \neq 1$, then the orthogonality relation \eqref{orthogonality} of characters implies that $$\sum\limits_{x\in\FF_q}\xi_1(x) \xi_2(x^2) = \sum\limits_{x\in\FF_q} \xi_1(x) = 0,$$ while if $\xi_2\neq 1$, then the multiplicative property $\xi_i(x + y) = \xi_i(x) \xi_i(y)$ allows us to deduce
\begin{align*}
    \left|\sum_{x\in\FF_q} \xi_1(x) \xi_2(x^2)\right|^2 = \sum_{x,h\in\FF_q} \xi_1(h) \xi_2(2hx + h^2) \leq \sum_{h\in\FF_q} \left|\sum_{x\in\FF_q} \xi_2(2hx)\right| = q + \sum_{h\in \FF_q\setminus{\{0\}}} \left|\sum_{x\in\FF_q} \xi_2(2hx)\right|.
\end{align*}
We then use the orthogonality to infer that for every $h\neq 0$, the sum of the nontrivial character $\xi_{2,h}(x) = \xi_2(2hx)$ over $\FF_q$ vanishes, and so $|\sum_{x\in\FF_q}\xi_1(x) \xi_2(x^2)|=q^\frac{1}{2}$ in this case (the assumption $p>2$ on the characteristic of $\FF_q$ is necessary for $\xi_{2, h}$ to be nontrivial). 
Hence $\widehat{1_\Pi}(\xi_1, \xi_2) = \frac{1}{q}$ if $\xi_1 = \xi_2 = 1$ and
\begin{align}\label{Fourier decay of parabola}
    |\widehat{1_\Pi}(\xi_1, \xi_2)|\leq q^{-\frac{3}{2}}
\end{align}
otherwise. Using the Fourier inversion formula, we deduce that
\begin{align}\label{eq in fin fields}
    \nonumber 
    \sum_{x, y, z\in\FF_q} f(x, y) g(x + z, y+z^2) &= \sum_{\substack{x, y, z,\\ u\in\FF_q}} f(x, y) g(x + z, y+u) 1_{\Pi}(z, u)\\
    &= \sum_{\xi_1, \xi_2\in\widehat{\FF_q}} \widehat{1_\Pi}(\xi_1, \xi_2) \sum_{\substack{x, y, z,\\ u\in\FF_q}} f(x, y) g(x + z, y+u) \xi_1(-z)\xi_2(-u).
\end{align}
Splitting \eqref{eq in fin fields} into  the term $(\xi_1, \xi_2) = (1, 1)$ and the sum over the other terms, we obtain the equality
\begin{align*}
     \sum_{x, y, z\in\FF_q} f(x, y) g(x + z, y+z^2) &= \frac{1}{q}\sum_{x,y\in\FF_q}f(x,y) \sum_{z, u\in\FF_q}g(z, u)\\
     &+ \sum_{\substack{\xi_1, \xi_2\in \widehat{\FF_q}\\ (\xi_1, \xi_2)\neq (1,1)}} \widehat{1_\Pi}(\xi_1, \xi_2) \sum_{\substack{x, y, z,\\ u\in\FF_q}} f(x, y) g(x + z, y+u) \xi_1(-z)\xi_2(-u).
\end{align*}
In the remainder of the proof, we shall bound the contributions of the terms with $(\xi_1, \xi_2)\neq (1,1)$ by invoking \eqref{Fourier decay of parabola}. Noting that 
\begin{align*}
    &\sum_{\substack{x, y, z,\\ u\in\FF_q}} f(x, y) g(x + z, y+u) \xi_1(-z)\xi_2(-u)\\
    &= \sum_{x,y\in\FF_q}f(x,y)\xi_1(x)\xi_2(y)\sum_{z, u\in\FF_q} g(x + z, y+u) \xi_1(-x-z)\xi_2(-y-u) =q^4 \widehat{f}(\xi_1, \xi_2) \overline{\widehat{g}(\xi_1, \xi_2)},
\end{align*}
we apply (\ref{Fourier decay of parabola}), the Cauchy-Schwarz inequality and Parseval identity to deduce that
\begin{align*}
    &\left|\sum_{\substack{\xi_1, \xi_2\in \widehat{\FF_q}\\ (\xi_1, \xi_2)\neq (1,1)}} \widehat{1_\Pi}(\xi_1, \xi_2) \sum_{x, y, z, u} f(x, y) g(x + z, y+u) \xi_1(-z)\xi_2(-u) \right|\\
    &\leq q^{\frac{5}{2}}\sum_{\substack{\xi_1, \xi_2\in \widehat{\FF_q}\\ (\xi_1, \xi_2)\neq (1,1)}} |\widehat{f}(\xi_1, \xi_2)|\cdot |\widehat{g}(\xi_1, \xi_2)| \leq q^{\frac{5}{2}} \cdot \|\hat{f}\|_2 \cdot \left\|\hat{g}\right\|_2 = q^{\frac{5}{2}}\cdot \|f\|_2 \cdot \left\|g\right\|_2.
\end{align*}
Hence 
\begin{align*}
        \left|\sum_{x, y, z\in\FF_q} f(x, y) g(x + z, y+z^2) - \frac{1}{q}\sum_{x,y\in\FF_q} f(x,y) \sum_{x,y\in\FF_q} g(x,y) \right|\leq q^\frac{5}{2}\|f\|_2 \|g\|_2,
\end{align*}
as claimed. 
\end{proof}
\bibliographystyle{plain}
\bibliography{references}

\end{document}